\documentclass{amsart}

\usepackage{amsmath,amssymb,amsthm,amscd,amsopn,amsfonts,amsxtra}
\usepackage{latexsym,array,verbatim,mathtools,tensor,hyperref}
\usepackage[mathscr]{eucal}
\usepackage{eucal} 

\usepackage
{graphicx}
\usepackage{psfrag,color,float,
caption,subcaption}
\restylefloat{figure}
\usepackage{fancybox,shadow,epic}
\usepackage{epsfig}
 \usepackage{mathrsfs}
\usepackage{calligra,longtable,tabularx}
\usepackage[T1]{fontenc}

\usepackage[titletoc,toc,title]{appendix}

\newtheorem{theorem}{Theorem}
\newtheorem{lemma}{Lemma}

\newtheorem{assumption}{Assumption}

\theoremstyle{definition}
\newtheorem{definition}{Definition}

\theoremstyle{remark}

\definecolor{orange}{rgb}{0.995, 0.75, 0.35}
\definecolor{purple}{rgb}{0.7, 0.2, 0.5}
\definecolor{royalblue}{rgb}{0.2, 0.7, 0.8}
\definecolor{darkgreen}{rgb}{0.2,0.725,0.25}

\def\al{\alpha}
\def\de{\delta}
\def\eps{\epsilon}
\def\ga{\gamma}
\def\veps{\varepsilon}

\def\om{\omega}
\def\De{\Delta}

\def\Om{\Omega}

\def\br{{\bf r}}
\def\curl{\mathrm{curl}}
\def\dive{\mathrm{div}\,}

\def\iy{\infty}
\def\cL{\mathcal{L}}

\def\pa{\partial}

\def\sH{\mathscr{H}}
\def\To{\Rightarrow}

\newcommand{\la}{\langle}
\newcommand{\ra}{\rangle}

\newcommand{\inv}{^{-1}}
\newcommand{\td}{\tilde}

\newcommand{\rks}{{\bf Remarks:}\ \ }
\newcommand{\rk}{{\bf Remark.}\ \ }

\newcommand{\nd}{\noindent}

\newcommand{\R}{\mathbb{R}}

\begin{document}
\title[Nonlinear Schr\"odinger Equations for BEC]{Nonlinear Schr\"odinger Equations for Bose-Einstein Condensates}

\subjclass[2010]{35Q55, 65M70}
\keywords{nonlinear Schr\"odinger equation, BEC, electromagnetic potential}

\author{Luigi Galati}
\author{Shijun Zheng}
\address[Luigi Galati and Shijun Zheng]
{Department of Mathematical Sciences\\
        Georgia Southern University\\
         Statesboro, GA 30460-8093
         }

\begin{abstract}
The Gross-Pitaevskii equation, or more generally the nonlinear Schr\"odinger equation, models the Bose-Einstein condensates in a macroscopic gaseous superfluid wave-matter state in ultra-cold temperature.
We provide analytical study of the NLS  with $L^2$ initial data 
in order to understand propagation of the defocusing and focusing waves for the BEC mechanism in the presence of  electromagnetic fields. Numerical simulations are performed for the two-dimensional 
 GPE with anisotropic quadratic potentials. 




\end{abstract}



\maketitle




\section{Introduction}  
 Consider the nonlinear Schr\"odinger  equation (NLS)
\begin{align}\label{e:u_A.V-mu.p}
&iu_t =-\frac12(\nabla-iA)^2 u+Vu+\mu |u|^{p-1}u \qquad (t,x)\in \R^{1+n}\\
&u(0,x)=u_0 ,
\end{align}
where $1\le p<\iy$, $\mu\in\R$, $V:\R^n\to \R$ induces the electric field $-\nabla V$, and $A=(A_1,\dots,A_n): \R^n\to \R^n$ induces the magnetic field
 $B=\nabla\wedge A=(\pa_jA_k-\pa_kA_j)_{n\times n}$. 
 Denote
$\De_A=\nabla_A^2:=\sum_{j=1}^n (\frac{\pa}{\pa x_j}-iA_j)^2$.
Then  $\cL:=-\frac12\De_A+V$ is an essentially selfadjoint Schr\"odinger operator with an electromagnetic potential $(A,V)$ that is gauge-invariant.
The nonlinear term $F(u)=\mu |u|^{p-1}u$ has the 
property $\Im(\bar{u}F(u))=0$.  

The physical significance for NLS in a magnetic field is well-known in   
  nonlinear optics and Bose-Einstein condensate (BEC), 
where the magnetic structure is involved in scattering, 
superfluid, quantized vortices 
as well as
DNLS  in plasma physics \cite{MAHHWC99,Gross1961,Pitae1961}. 
There have been produced BEC where Bosons, Femions or other quasi-particles are trapped with atomic lasers in order to observe the macroscopic coherent wave matter in  ultra-cold temperature.

The Hamiltonian
$H:=\int \frac{\hbar^2}{2m} |\nabla_A u|^2+ \frac{2\mu}{p+1}|u|^{p+1}$
  generates the nonlinear system in  \eqref{e:u_A.V-mu.p}:
\begin{align}
&i\hbar \frac{\pa u}{\pa t}=\frac{\de H}{\de \bar{u}}\label{e:hu-hamilton.variation}
\end{align}
with $\hbar=m=1$ ($\hbar$ being the Planck constant and $m$ the mass of a particle),
where we note that the adjoint of the covariant gradient $\nabla_A$ is $-\nabla_A$.  
 When $p=3$, we obtain the Gross-Pitaevskii equation (GPE), 
 which is regarded as a 
{Ginzburg-Landau model} in string theory.
 In general, the operator $\nabla_A^2=(\nabla-iA)^2$ contains components of both the (trapping) angular momentum and  (attractive/repulsive) potential that can affect the dispersion of NLS.
The equations $\eqref{e:u_A.V-mu.p}=\eqref{e:hu-hamilton.variation}$ play the role of Newton's law in classical mechanics \cite{Sq08}.


 In the state of superfluid, the gaseous BEC has the vortices phenomenon which arises from (in the focusing case) the bound states of the form $u=e^{i\ga t}Q$, $Q(x)= e^{i m\theta }R_m(r)$ being an excited state. 
Another situation where it appears is when we test or manipulate the BEC by a magnetic trap with rotation. The  wave function for the condensate is the solution of the following NLS  
\begin{align}\label{e:nlsV-L_Om}
&i\pa_t u =-\frac{1}{2}\nabla^2u+ \mu |u|^{p-1}u + \td{V}u-\Om\cdot Lu ,
\end{align}
where the rotation term $\Om\cdot L=-i\Om\cdot(x\wedge \nabla)$, $\Om=(\om_1,\dots,\om_n)\in \R^n$
and $L$ denotes the angular momentum operator  
\cite{HHsiaoL07b, ChC07,AMS10}.
Comparing \eqref{e:u_A.V-mu.p} and \eqref{e:nlsV-L_Om}
one finds  
\begin{align}
&\td{V}(x)= \frac12|A(x)|^2+V(x)+ \frac{i}{2} \dive A(x) \label{e:tdV-AV} \\
&i\Om\cdot (x\wedge\nabla)= iA(x)\cdot \nabla .\label{e:Om-A32}
\end{align}
If $n=3$, then  $\dive A=0$ with $A=(\om_2 x_3-\om_3  x_2,
\om_3 x_1-\om_1x_3, \om_1 x_2-\om_2x_1)$. 
As a $2$-form  $B=\curl \;A$ 
 is constant. 
A simple calculation shows that the tangential component of $B$ is
\begin{align*} &B_{\tau}=B\cdot \frac{\textbf{r}}{r}=\begin{pmatrix} 0& 2\om_3& -2\om_2\\
-2\om_3&0&2\om_1\\
2\om_2&-2\om_1 &0
\end{pmatrix} \cdot\frac{{\bf r}}{r}
= -\frac{2}{r}A ,
\end{align*}
where ${\bf r}=(x_1,x_2,x_3)$ and $r=|{\bf r}|$.
This tells that $B$ is a ``trapping" field whenever $\Om\ne 0$. Heuristically $B_\tau\ne 0$ indicates an obstruction to the dispersion 
\cite{F09}.  In $\R^3$, if  the Coulomb gauge $\dive A=0$, then one can recover $A$ as a ``weighted wedge" of $x$ and $B$ 
\begin{align*}
&A(x)=\frac{1}{4\pi}\int \frac{x-y}{|x-y|^3}\wedge  B(y)dy.
\end{align*}



Let $\td{V}(x)=\frac{1}{2}\sum_{j} \ga_j^2 x_j^2$ and $|\ga|=(\sum_j \ga_j^2)^{1/2}$. 
When $|\Om| \ll |\ga|$, the rotation action is negligible, and the potential $\td{V}$ is more predominant so that
$A\approx 0$, $V\approx \td{V}=\frac12 \sum_j \ga_j^2 x_j^2$. In this case  we anticipate trapping.
When $|\Om| \gg |\ga|$, the rotation is much stronger than $\td{V}$ so that the effect of $\td{V}\approx 0\To$
$V\approx -|A|^2/2$. 
This suggests that the wave function of a rotating BEC may be subject to an anisotropic repulsive potential.
In this case the dispersion might hold global in time so that the (focusing) nonlinearity turns to be ``short range"
 resulting in 
  scattering \cite{LiuT04,Car05,AMS10}.
Geometrically, the $x^2$ potential affects the wave like the trapping condition, which is stable, on a spherical portion of a manifold,
while the $-x^2$ potential affects the wave like the scattering (non-trapping) condition on a hyperbolic portion of a manifold, which can be unstable locally in time but stable global in time.




The main analytical interest of this paper is to study the $L^2$ solution of \eqref{e:u_A.V-mu.p} under the following assumptions on $A$ and $V$ throughout this section.

\begin{assumption}\label{a:A-lin-V-quadr}
 Let $A_j$ and $V$ be real-valued and belong to $C^\iy(\R^n)$.  Let $V$ be bounded from below. \\
   Assume $A=(A_j)_{j=1}^n$ is sublinear and $V$ subquadratic, namely, 
\begin{align*}
&\pa^\al A_j(x)=O(1),  \quad\forall \,|\al| \ge 1, \\
&\pa^\al V(x)=O(1),   \quad\forall \,|\al| \ge 2 
\end{align*}
as $|x|\to \iy$.  In addition, assume there exists some $\veps>0$ such that for all $|\al | \ge 1$
\begin{align*} |\pa^\al B(x)|\le c_\al \la x\ra^{-1-\veps} ,
\end{align*}
where  $B= (b_{jk})_{n\times n}$, $b_{jk}= \pa_j A_k- \pa_k A_j$.
\end{assumption}

Define the $\cL$-Sobolev space $\sH^{s,r}:
=\{u: \nabla^s u\in L^r, \la x\ra^s u\in L^r  \}$, where $\la x\ra=(1+|x|^2)^{1/2}$. When $r=2$,
we will also use the abbreviation $\sH^1=\sH^{1,2}$.
For $u_0$ in $\sH^1$, local wellposedness of \eqref{e:u_A.V-mu.p} was  proven for $1\le p<1+4/(n-2)$
  e.g., in \cite{deB91,Na01,Mi08} based on the fundamental solution constructed in \cite{Ya91}.  The $\sH^s$ subcritical result was considered in \cite{Z12} for $1\le p<1+4/(n-2s)$.
When $s=1$,  the following are known: 
 Let $u_0\in \sH^1$, $r=p+1$ and  $q=\frac{4p+4}{n(p-1)}$.
\begin{enumerate}
\item Let $1\le p< 1+4/(n-2)$. Then in the defocusing case $\mu>0$, \eqref{e:u_A.V-mu.p} has an $\sH^1$-bounded  global solution in
$C(\R,\sH^1)\cap L^q_{loc}(\R,\sH^{1,r})$.
 In the focusing case $\mu<0$, if $1\le p<  1+4/d$, then \eqref{e:u_A.V-mu.p} has an $\sH^1$-bounded global solution in
$C(\R,\sH^1)\cap L^q_{loc}(\R,\sH^{1,r})$.
\item Let $p=1+4/(n-2)$, $n\ge 3$.
  If $\Vert u_0\Vert_{\sH^1}< \veps$ for some $\veps=\veps(n,|\mu|)$ sufficiently small,
then \eqref{e:u_A.V-mu.p} has a unique local solution in $C((-T,T),\sH^1)\cap L^q((-T,T),\sH^{1,r})$ for some $T>0$.
\end{enumerate}

In two and three dimensions similar results on the $\sH^1$ subcritical problem for  \eqref{e:nlsV-L_Om}  have been obtained in \cite{HHsiaoL07a, HHsiaoL07b, AMS10}.
The main theorem (Theorem \ref{t:gwp-L2smalldata}) we state below is the global wellposedness of \eqref{e:u_A.V-mu.p} for $L^2$ initial data
by virtue of the maximal Strichartz norm. 
This strengthens Theorem  3.3  in \cite{Z12}.

\begin{definition}\label{de:adm.pair-qr} We call $(q,r)=(q,r,n)$ an admissible pair if $q,r\in [2,\iy]$ satisfy
 $(q,r,n)\neq (2,\iy,2)$ and
\begin{align*}
\frac{2}{q}+\frac{n}{r}=\frac{n}{2}.
\end{align*}   
\end{definition}


\begin{definition}\label{de:max-stri.qr} Let $I\subset \R$ be an interval. 
The Strichartz space $S^0(I):=S^0(I\times \R^n)$ is a Banach space consisting of functions in $\cap_{(q,r)\;admissible}L^qL^r(I\times\R^n)$ satisfying
\begin{align*} \Vert u \Vert_{S^0(I)}:= \sup_{(q,r)\; admissible} \Vert u\Vert_{L^qL^r(I\times\R^n)}<\iy.
\end{align*}
 Define $ N^0(I)$ to be the linear span of $ \cup_{(q,r)\;admissible} L^{q'}L^{r'}(I\times \R^n)$, where $q'=q/(q-1)$ is the H\"older conjugate of $q$.  If $n\ge 3$, the admissible pairs include the endpoint pair $(2,\frac{2n}{n-2})$, which allows us to
identify $S^0(I)$  with $L^\iy L^2\cap L^2L^{\frac{2n}{n-2}}(I\times\R^n)$ through interpolation.
 In this case
$ N^0(I)=L^1L^2+ L^2L^{\frac{2n}{n+2}}(I\times \R^n)$ and $N^0(I)$ is endowed with the norm
\begin{align*} \Vert f\Vert_{N^0(I)}=\min_{f=f_1+f_2} \left(\Vert f_1\Vert_{L^1L^2(I\times\R^n)} +\Vert f_2\Vert_{L^2L^{\frac{2n}{n+2}}(I\times \R^n)} \right),
\end{align*}
where the infimum is taken over all possible $f_1\in L^1L^2(I\times \R^n)$ and $f_2\in L^2L^{\frac{2n}{n+2}}(I\times \R^n)$ such that $f=f_1+f_2$ \cite{LiZh10}.
\end{definition} 

\begin{theorem}\label{t:gwp-L2smalldata} Let $A$  and $V$ satisfy the conditions in Assumption \ref{a:A-lin-V-quadr}.  Suppose $u_0\in L^2(\R^n)$.
\begin{enumerate}
\item 
If $1\le p< 1+4/n$, then equation \eqref{e:u_A.V-mu.p}
has a unique solution $u$ in $ C(\R, L^2(\R^n))\cap S^0_{loc}(\R\times \R^n)$.
Furthermore, for any $R>0$ there exists $T_R>0$ such that  the flow $u_0\mapsto u$ is Lipschitz continuous from $ \mathcal{B}_R$ into $S^0((-T_R,T_R))$.
\item If  $p=1+4/n$, then there exists an $\veps>0$ such that
  $\Vert u_0\Vert_2<\veps$ implies that
 equation \eqref{e:u_A.V-mu.p}  has a unique solution $u$ in $C(\R, L^2(\R^n))\cap S^0_{loc}(\R\times \R^n)$.
 The flow $u_0\mapsto u$ is Lipschitz continuous from $\mathscr{B}_{\eps/2}$ into $S^0((-T_0,T_0))$.
\end{enumerate}
In both cases,  it holds  that for all $T>0$, 
\begin{align*}
\Vert u\Vert_{S^0( (-T,T))}\le c T .
\end{align*} 
 In the above, $\mathscr{B}_R:=\{u: \Vert u\Vert_2\le R\}$, $T_0=T_0(A,V)$, and
 $\veps$ and $c$ are constants depending on $n$, $\mu$ and $\Vert u_0\Vert_2$ only.
 \end{theorem}

In the focusing ($\mu<0$), $L^2$ critical or supercritical but energy subcritical regime $1+4/n\le p<1+4/(n-2)$, there can occur finite time blowup solutions for \eqref{e:u_A.V-mu.p}, see e.g., \cite{deB91,Car02c,Sq08}.  
Such situation is more complicated, where the occurrence of wave collapse is equivalent to the existence of soliton,
which depends on
the interaction between linear and nonlinear energies, the expectation of momentum as well as the profile of the initial data.
For the rotating problem \eqref{e:nlsV-L_Om}, wave collapse can occur for either cases
where $|\Om|\ll |\ga|$ or $|\Om|\gg |\ga|$. In \cite{AMS10} blowup conditions are given in terms of $(\Om\cdot L)V$.
More recently, Garcia \cite{Gar12} obtained a general blowup criteria for \eqref{e:u_A.V-mu.p} based on
 spectral properties of $A$ and $V$.

It is desirable to observe numerical results that can experimentally verify the theory.
In Section 4 
we apply the  Strang splitting scheme to find numerical solutions
for the GPE \eqref{e:u_A.V-mu.p} in 2D (a cubic NLS)
where we take $A=0$ and $V(x_1,x_2)=\frac12\sum_{j=1}^2 \de_j \ga_j^2 x_j^2$,
$\de_j\in \{\pm 1\}$.   
Our algorithm and implementations are based on time-splitting Fourier-spectral methods developed in \cite{
 bao2002time} 
 and GPELab \cite{GPElab}.
Such scheme is stable and has higher accuracy under appropriate conditions on $V$ and initial data,
see 
\cite{Lubich2008splitting, LuM13}.
 Numerical schemes typically use spectral or pseudo-spectral method to approximate the solution
 by discretizing spacial dimensions and then advancing a time step,   
while physicists  have used e.g., Crank-Nicholson method 
via 
Lagrangian variational techniques \cite{ruprecht95, BEdKSC2012}. 

The organization of the remaining of the paper is as follows. In Section 2, the time dependent Gross-Pitaevskii equation, in particular the electromagnetic GPE,  is introduced and  formally derived as mean field approximation for the $N$-particle state of the BEC.
 In Section 3, we prove
 Theorem \ref{t:gwp-L2smalldata}, mainly in the $L^2$-critical case, concerning global wellposedness of \eqref{e:u_A.V-mu.p}.
In Section 4, 
 we present numeral simulations to  illustrate  the focusing and defocusing nonlinear effects on the wave function of BEC subject to various anisotropic  harmonic potentials.  

\section{Formal derivation of the Gross-Pitaevskii equation}\label{s:form.deriv-GPE}
 In the early stage of  quantum mechanics there arose questions concerning
 fundamental aspects  
of decoherence and measurement theory as well as understanding the correlation between classical and quantum scattering models.   In 1924, Satyendra Nath Bose published a paper describing the statistical nature of light \cite{BOSE24}.  Using Bose's paper, Albert Einstein predicted a phase transition in a gas of noninteracting atoms could occur due to these quantum statistical effects \cite{Ein24,Ein25}.  This phase transition period, Bose-Einsten Condensation, would allow for a macroscopic number of non-interacting bosons to simultaneously occupy the same quantum state of lowest energy. 

It wasn't until 1938, with the discovery of superfluidity in liquid helium, that F. London conjectured that this superfluidity may be one of the first manifestations of BEC. 
 The real breakthrough came in 1995,  when the BEC were produced from a vapor of
 rubidium, and of sodium atoms \cite{And1995,Davis1995}. 

The Gross-Pitaevskii equation \eqref{e:u_A.V-mu.p}, $p=3$  
describes  the macroscopic wave functions $u$ of the condensate 
in the presence of  the magnetic and electric potentials $A$ and $V$.  
The nonlinear term results from the mean field interaction between atoms. 
   The constant $\mu$ accounts for the attractive ($\mu<0$)  or repulsive ($\mu>0$) interaction,
   whose sign depends on the chemical elements.

Nowadays BEC can be simulated in the computer and the lab. 
The rotating BEC, for instance, involves  the decoherence $\leftrightarrow$ coherence phase.
The angular momentum operator breaks up the beams, hence split the spectral lines when performing the experiment on silver atoms in normal state.
It can help create quasi-particles so to manipulate or observe not only the macroscopic atoms, but also individual particle.
There are potential  applications in higher degree precision for measurement, navigation, computing and communications.

\subsection{A formal derivation} We follow a mean-field approach to derive the time-dependent GPE for the $N$-body system of bosons. 
At ultra low temperatures, all bosons exist in identical single-particle state $\phi({\bf r})$, ${\bf r}\in \R^3$ and so  we can write the wave function of the $N$-particle system as 
\begin{equation}\label{manybodied}
\Psi({\bf r}_1, {\bf r}_2, \ldots ,{\bf r}_N) = \prod_{i=1}^N \phi({\bf r}_i).
\end{equation}
The single-particle wave function $\phi({\bf r})$ obeys the typical normalization condition
\begin{equation*}
\int_{\mathbb{R}^3} |\phi({\bf r})|^2 d\mathbf{r} = 1.
\end{equation*}
Due to the fact that we are dealing with dilute gases, the distance between any two particles in positions ${\bf r}$ and ${\bf r}'$ is such that the only interaction term is  $U_0 \delta ({\bf r} - {\bf r}')$, where $\de$ is the usual Dirac function and $U_0 = \frac{4\pi \hbar^2 a}{m}$ is the strength of effective contact interaction ($a$ being the scattering length). 
    Thus the Hamiltonian reads
\begin{equation*}
H_N = \sum_{i=1}^N \left[ \frac{{\bf p}_i^2}{2m} + V({\bf r}_i) \right] + U_0 \sum_{i<j} \delta ({\bf r}_i - {\bf r}_j) ,
\end{equation*}
where ${\bf p}=-i\hbar\nabla$ stands for the momentum and
$V({\bf r})$ the external potential.  Meanwhile the $N$-state (\ref{manybodied}) has energy
\begin{equation}\label{energy}
E_N = N \int_{\R^3} \left[ \frac{\hbar^2}{2m}| \nabla \phi({\bf r})|^2 + V({\bf r})|\phi({\bf r})|^2 + \frac{(N-1)}{2} U_0|\phi({\bf r})| ^4\right]d{\bf r} ,
\end{equation}
where the nonlinear energy term is attributed to the inherent self-interaction and interaction between a pair of  bosons on the same state
\begin{align*}
&\int_{\R^6} U_0\de({\bf r}_i-{\bf r}_j)
\la \phi({\bf r}_i)\vert \de({\bf r}_i-{\bf r}_i')\phi({\bf r}'_i) \ra  \la \phi({\bf r}_j)\vert  \de({\bf r}_j-{\bf r}'_j)\phi({\bf r}'_j)\ra d{\bf r}_i  d{\bf r}_j \notag\\
=&\int_{\R^6} U_0\de({\bf r}_i-{\bf r}_j)
| \phi({\bf r}_i) |^2   | \phi({\bf r}_j)|^2  d{\bf r}_i  d{\bf r}_j =\int_{\R^3} U_0 | \phi({\bf r}_i) |^4  d{\bf r}_i \,.\label{e:delta-phi.N}
\end{align*}
This is equivalent to an expression in terms of the expectation of the collision contact. 

Introduce the wave function for the condensed state
\begin{equation*}
\psi({\bf r}) = N^{1/2}\phi({\bf r})
\end{equation*}
so that  $N=\int |\psi|^2 d{\bf r} $. 
By a variation argument for $E_N$, similar to \eqref{e:hu-hamilton.variation} we formulate the GPE as $N\to \iy$
\begin{equation}\label{e:condwavefunc}
i\hbar \frac{\pa\psi}{\pa t} = -\frac{\hbar^2}{2m} \nabla^2 \psi + V({\bf r})\psi + U_0|\psi|^2 \psi .
\end{equation}

\begin{proof}[Derivation of magnetic NLS]  In a similar way we can formally derive \eqref{e:u_A.V-mu.p} for $p=3$.
Let 
$A \in L_{loc}^2(\R^3,\R^3)$, $V: \R^3\to \R$. 
Assume an $N$-particle weakly interacting condensate of non-relativistic bosons without spin 
 in the mean field.
The Hamiltonian in the electromagnetic 
frame has the form on $\R^{3N}$
\begin{align*}&H_{N}= \sum_{\iota=1}^N \left(-\frac{\hbar}{2m}\nabla_{A,\iota}^2+ V(\br_\iota)\right) + \mu \sum_{\iota<j}^N g(\br_\iota-\br_j) ,
\end{align*}
 where $\nabla_A=\nabla-iA $ is the covariant gradient on $\R^3$,
 $V$ 
 represents the external potential, 
$\mu g$ the inherent potential for a two-body bosons, that is, the interaction between two particles 
is given by $\mu g(\br-\br')$.
 Using  the fact that 
the expectation at $(t,\br)$
of the interaction from the $\iota$-th particle is
$\mu\int_{\R^3} g(\br-\br_\iota) |\psi(t,\br_\iota)|^2 d\br_\iota$  
we arrive at the GPE that decries the wave function of the condensate 
 \begin{align*}&i\hbar\frac{\pa}{\pa t}\psi= -\frac{\hbar^2}{2m}\De_{A}\psi + V\psi + \mu  (g*|\psi|^2)\psi.
 \end{align*}
In the case $g=\de$ where only local contact interaction from collision is accounted for while other interactions are neglected in a  dilute gas, the equation  becomes the standard magnetic cubic NLS.
\end{proof}

\nd\rk The derivation above relies on the fact that
 the $N$ particles of a dilute gas are condensed in the same state for which the wave function minimized
the 
energy.
 The note \cite{Gr09} contains derivation and discussions of the magnetic GPE in the physical setting.
For rigorous derivation of the mean field limit of the $N$-particle coherent state as $N\to  \iy$ as well as $t\to \iy$ involving ground state trapping and scattering (dispersion) we refer to \cite{FrL04, ESY07a}.

\paragraph{GPE with harmonic potential and angular momentum}  In \eqref{e:condwavefunc}, 
 $|\psi(t,x)|^2$ denotes the probability density of the condensate at $(t,x)$.
The coefficient $\mu$ measures the strength of interaction  and depends on a quantity called the $s$-scattering length.
It has positive sign  (defocusing) for $\tensor*[^{87}]{Rb}{}$,
$\tensor*[^{23}]{Na}{}$, $\tensor*[^1]{H}{}$ atoms, but
negative sign (focusing) for $\tensor*[^7]{Li}{}$,
$\tensor*[^{85}]{Rb}{}$, $\tensor*[^{133}]{Cs}{}$ 
\cite{WTs98,Car02c}.  The typical example $V=\frac12\sum_j \ga_j^2 x_j^2$
 represents an external
trapping potential imposed by a system of laser beams, 
 where $\ga_1,\gamma_2, \gamma_3$ are the magnitudes of the frequencies of the oscillator in three directions.
It works as  
 an anisotropic trap that allows one to observe the behavior of macroscopic waves
 traveling along a waveguide with varying width or excitations when a BEC is released from a trap.


With the addition of a rotation term  
 we arrive at the GPE in \eqref{e:nlsV-L_Om}.
 This equation is viewed as a conservation for the angular momentum on a quantum level that involves  Newton's law and Lorentz force  where the magnetic field is divergence free.
The momentum operator $L_\Om:=i\Om \cdot (x\wedge \nabla)$
with non-vanishing angular velocity $\Om$ 
  gives rise to  vortex lattices in a condensate that supports it in turn,  e.g., one can obtain the vortex lattices of a 
 BEC by setting the $Na$ condensate in rotation  using laser beams \cite{AKe02,MuH02}.

 The study of BEC as a rotating superfluid 
 leads to the quantization of circulation and quantized vortices. Physically this makes it impossible for a superfluid to rotate as a rigid body:\,In order to rotate, it must swirl \cite{AKe02}.  
 The existence of quantized vortices with such particular pattern  has  been verified by experiments and numerics,  see e.g.,
 \cite{MAHHWC99,
 BR2013} and  \cite{
 ADu01,ChC07}. 
They can be observed in a condensate with either
optical  or magnetic traps.

\section{The $L^2$ solution using maximal Strichartz norm}
  This section is devoted to the proof of Theorem \ref{t:gwp-L2smalldata}. We let $A$ and $V$ satisfy Assumption \ref{a:A-lin-V-quadr}.
A priori, note that equation \eqref{e:u_A.V-mu.p} has the conservation of mass and 
 energy on its lifespan
\begin{align}
&\Vert u(t)\Vert_2=\Vert u_0\Vert_2 \label{e:L2-conserv.u}\\
&E(t):= \int (\cL u) \bar{u} dx+\frac{2\mu}{p+1}\int |u|^{p+1} dx\notag\\
=&\la \cL u,u \ra+  \frac{2\mu}{p+1}\Vert u\Vert_{p+1}^{p+1}=E(0) .   \label{e:E(t)-H.conserv}
\end{align}

Let $u$ and $F$ be $L^2\cap L^r(\R^n)$-valued functions in $t \in I$,  $I$ an interval in $\R$. If $u$ solves
\begin{align}\label{e:iu-L.F_vec} &i u_t=\cL u+F(t),\qquad
 u(0)=u_0\in L^2(\R^n),
\end{align}
then the solution can be expressed in an integral form according to Duhamel principle
\begin{align}  &u=(i\pa_t-\cL)\inv F \notag\\
=&e^{-it\cL} u_0-i \int_0^t e^{-i(t-s)\cL} F(s)ds. \label{e:duhamel-U(t).u.F}
\end{align}

 From \cite{Ya91} we know  
  there exists $T_0$ such that
 for $0<|t|<T_0$ the propagator $U(t):=e^{-it\cL}$ is given as 
 \begin{align}\label{e:propagator-L-int}
 &U(t)f(x)= (2\pi it )^{-n/2}\int e^{iS(t,x,y)}a(t,x,y) f(y)dy, \end{align}
 where  $S(t,x,y)$ is a  real solution of the Hamilton-Jacobi equation,
both $S$ and $a$ are $C^1$ in $(t,x,y)$ and $C^\iy$ in $(x,y)$, with $|\pa_x^\al \pa_y^\beta a(t,x,y)|\le c_{\al\beta}$
for all $\al,\beta$.  Write $I:=I_{T_0}=[-T_0,T_0]$ and  $L^q L^r (I\times \R^n)=L_t^q  (I, L_x^r(\R^n))$.

\begin{lemma}[Strichartz estimates \cite{deB91,Z12}]\label{l:disp-Stri-L} If $A$ and $V$ satisfy Assumption \ref{a:A-lin-V-quadr},
then we have for $I=[-T_0,T_0]$, 
there exist  constants $c_q,c_{q,\td{q}}$ 
such that  \begin{itemize}
\item[] \begin{equation}\label{e:u-pq-f_L2}
\|U(t)f \|_{L^{q}L^r(I\times\R^n )} \le c_{q}\|f\|_{2 }
\end{equation}
\item[] \begin{equation}\label{e:inhom-stri_L2}
\|\int_0^t U(t-s)F(s,\cdot)ds \|_{L^qL^r(I\times\R^n )}\le  c_{q,\td{q}}\|F\|_{L^{\td{q}'}L^{\td{r}'}(I\times\R^n )} ,
\end{equation}
where  $(q,r), (\td{q}, \td{r})$ are any admissible pairs, and $q'$ is the H\"older conjugate of $q$.
\end{itemize}
\end{lemma}

The Strichartz estimates 
yield the following lemma, consult \cite[Chapter 3]{Tao06}.
\begin{lemma}\label{l:S-u.N-F}
 Let $u$ be a solution of \eqref{e:iu-L.F_vec}. 
  Then for any admissible pairs $(q,r)$, $(\td{q},\td{r})$ as in Definition \ref{de:adm.pair-qr}
we have
\begin{align}
\Vert u\Vert_{L^q L^r (I\times \R^n) }\le c_{q, \tilde{q}}(\Vert u_0\Vert_2
+\Vert F\Vert_{L^{\td{q}'} L^{\td{r}'} (I\times \R^n) } ). \label{e:u_qr-F_qr-glob}
\end{align}
Moreover, \begin{align}
&\Vert u\Vert_{S^0(I) }\le c_n (\Vert u_0\Vert_2
+\Vert F\Vert_{N^0(I) } ). \label{e:u-F-S.N_I}
\end{align}
\end{lemma}

 Now we begin to prove part (2) in Theorem \ref{t:gwp-L2smalldata}.
\begin{proof}[Proof of (2) in Theorem \ref{t:gwp-L2smalldata}]
  (I)  Let 
  $p=1+4/n$.
According to  \eqref{e:u-F-S.N_I},  we have
\begin{align*}
\Vert u \Vert_{S^0(I)}\le c_n (\Vert u_0\Vert_{2}+ \Vert |u|^{\frac{4}{n}}u \Vert_{N^0(I)} ).
\end{align*}
Since $N^0(I)\supset \cup_{(q,r)\; admissible}L^qL^r(I\times \R^n)$ and $q=r={(2n+4)}/{n}$ are admissible,
it follows that 
\begin{align*}
&\Vert |u|^{\frac{4}{n}}u  \Vert_{N^0(I)} \le \Vert |u|^{\frac{4}{n}}u\Vert_{L^{(\frac{2n+4}{n} )'} (I \times \R^n)}\notag\\
=&\Vert |u|^{\frac{n+4}{n}} \Vert_{L^{\frac{2n+4}{n+4} } (I \times \R^n)}
 = \Vert u\Vert^{\frac{n+4}{n}}_{L^{\frac{2n+4}{n}}(I\times \R^n)}  .
\end{align*}
Hence we obtain
\begin{align}
\Vert u \Vert_{S^0(I)}\le  c_n(\Vert u_0\Vert_2+ \Vert u \Vert^{\frac{n+4}{n}}_{S^0(I)} ). \label{e:u_S-F_S}
\end{align}

(II) Let $\Vert u_0\Vert_2\le \veps:=\eta\ga=(2c_n)^{-1-n/4}\min(1, (5|\mu|)^{-n/4})$, where we choose $\eta=(2c_n)\inv$ and
$\ga= \min((2c_n)^{-n/4}, (10c_n|\mu|)^{-n/4})$.  In view of  \eqref{e:duhamel-U(t).u.F} we need to prove that the mapping
\begin{align}
 &\Phi(u):=
 U(t) u_0-i\mu \int_0^t U(t-s) (|u|^{p-1} u) ds\label{e:u-U(t)-u^p_duhamel}
\end{align}
is a contraction on the closed set $E_\ga=\{u\in S^0(I): \Vert u\Vert_{S^0(I)}\le \ga \}$.

(a)  In doing so first we show $\Phi$: $E_\ga\to E_\ga$. According to \eqref{e:u_S-F_S}
we have, for $u\in E_\ga$ 
\begin{align*}
  & \Vert \Phi(u)\Vert_{S^0(I)}\le c_{n}(\Vert u_0\Vert_{2}  +
  \Vert  u \Vert_{S^0(I)}^{\frac{n+4}{n}})\\
\le& c_n \eta \ga + c_n \Vert  u \Vert_{S^0(I)} \ga^{4/n}\le \frac{\ga}{2}+ \frac{\ga}{2}=\ga .
\end{align*}

(b) Then we show that $\Phi$ is contraction on $E_\ga$. Note the following inequality: For all $p>1$ 
\begin{align*}
& | |u|^{p-1} u -|v|^{p-1}v|  \le p (\max (|u|, |v|) )^{p-1} |u-v|\\
\le& p (|u|^{p-1}+ |v|^{p-1})|u-v|.
\end{align*}
 H\"older inequality gives, with $p=1+4/n$,
\begin{align*}
&  \Vert |u|^{4/n} (u-v)   \Vert_{L^{\frac{2n+4}{n+4}}(I\times \R^n)} \le \Vert u-v   \Vert_{L^{\frac{2n+4}{n}}(I\times\R^n)}  \Vert u   \Vert_{L^{\frac{2n+4}{n}}(I\times \R^n)}^{\frac{4}{n}}  .
\end{align*}
The same type of inequality holds with $|u|^{4/n} (u-v)$ replaced by $|v|^{4/n} (u-v)$.

Hence, applying Lemma \ref{l:S-u.N-F} we obtain, for $p=1+4/n$ and ${\td{q}}={\td{r}}=(2n+4)/n $
\begin{align*}
  &    \Vert \Phi(u)-\Phi(v)\Vert_{S^0(I )}\\
  = & \Vert \mu\int_0^t  e^{-i(t-s)\cL} ( |u|^{4/n} u-|v|^{4/n}v ) ds \Vert_{S^0(I)}\\
\le&  |\mu|c_{n}  \Vert   |u|^{4/n} u -|v|^{4/n}v  \Vert_{L^{\td{q}'} L^{\td{r}'}(I\times\R^n)} \\
\le&p|\mu|c_n\Vert u-v\Vert_{L^{\frac{2n+4}{n}}(I\times\R^n)}  (\Vert u\Vert_{L^{\frac{2n+4}{n}}(I\times\R^n)}^{\frac{4}{n}}+\Vert v\Vert_{L^{\frac{2n+4}{n}}(I\times\R^n)}^{\frac{4}{n}}).
\end{align*}
It follows that for $u,v\in E_\ga$ 
\begin{align*}
  &\Vert \Phi(u)-\Phi(v)\Vert_{S^0(I)}
   \le C_{n} \ga^{\frac{4}{n}} \Vert u-v \Vert_{L^{\frac{2n+4}{n}}(I\times\R^n)}
    \le \frac12\Vert u-v  \Vert_{S^0(I)}
\end{align*}
by the choice of $\ga$ above, where $C_n=2c_n p|\mu|$ . 
 Therefore we have proved that $\Phi$ has a fixed point in the set $E_\ga$.
  We conclude that if $\Vert u_0\Vert_2\le \veps$, there exists a unique solution
 $u$ in $L^\iy([-T_0,T_0],L^2(\R^n))\cap S^0([-T_0,T_0],\R^n)$. The global in time existence follows from the conservation of the $L^2$ norm \eqref{e:L2-conserv.u} by observing that $\veps$ only depends on $n$ and $\mu$.

 (III) 
The Lipschitz continuity is based on iteration of the contraction  $\Phi$,
see e.g., Proposition 1.38 
and Proposition 3.17 in \cite{Tao06}. 
\end{proof}

 \begin{proof}[Proof of (1) in Theorem \ref{t:gwp-L2smalldata}] Let $p<1+4/n$.
 The proof for the subcritical case follow the same line as for the critical case 
 but use the following:
 Choose $(q,r)=(\td{q},\td{r})=(\frac{4p+4}{n(p-1)}, p+1)$ to arrive at 
\begin{align*}
&\Vert Pu-Pv\Vert_{S^0(I)}
\le c_{n,p} |I |^\al
(4 \Vert u_0\Vert_2)^{p-1} \cdot  \Vert u-v\Vert_{L^q(I,L^r)}\\
\le& \frac12 \Vert u-v\Vert_{S^0(I)}, \end{align*}
if choosing $T=T(\Vert u_0\Vert_2)>0$ sufficiently small.
Here we notice that $\al=
 \frac{4- n(p-1)}{ 4}  >0\iff p< 1+4/n$.
\end{proof}

\nd\rks For  $A=V=0$ the analogous result was proven in \cite{Ts87a,CazW89,Caz03} using $L_t^qL_x^r$ norm.
The case where $A=0$ and $V$ is subquadratic or quadratic was treated in \cite{Oh89,Car02c,Car05}.
The proof presented here is a modification of the standard argument, see \cite{Tao06}.

 When $1+4/n\le p<1+4/(n-2)$, Carles  \cite[Theorem 1.4]{Car05} shows that if $A=0$ and $V= -\frac12|x|^2$ (more generally,  $V$ has a stronger repulsive component),  then global in time existence and scattering hold in $\sH^1$.
Carles' proof relies on global in time Strichartz estimate 
where the repulsive component of $V$ produces exponential decay for $U(t)$ that balances the confining force from its attractive component   
to control the nonlinear effects.
In the energy critical case $p=1+4/(n-2)$, Killip, Visan and Zhang proved the GWP and scattering
  for radial initial data in $\sH^1$ \cite{Zh07,KVZh09}. 

\section{Numerical simulations for GPE}\label{s:num-beclabV}
  The Strang splitting method 
  \cite{strang1968construction}  deals with  
  hyperbolic model problems with second order accuracy for finite difference schemes,
which initially appeared in \cite{lax1964difference}. 
 For a general nonlinear system one can write
\[u_t = c(t,{\bf x},D^{\alpha}u) = a(t,{\bf x},D^{\alpha}u) + b(t,{\bf x},D^{\alpha}u )\]
to obtain  the following two equations for which $u=v+w$ and 
\begin{align*}
v_t = a(t,{\bf x},D^{\alpha}u), \quad w_t= b(t,{\bf x},D^{\alpha}u).
\end{align*}
 In the NLS case
 this method also has second order  stability \cite{Lubich2008splitting}. 
Since the splitting scheme can preserve the structure of the PDE,
it also preserves the same conservation quantities \eqref{e:L2-conserv.u} and \eqref{e:E(t)-H.conserv} for the numerical solution as well as the analytic solution.

In this section we apply the Strang splitting algorithm to find numerical solution of the GPE in two dimensions.
By truncation we consider the following equation defined on the rectangle $R:=[a,b]\times [c,d]$ with periodic boundary conditions
 \begin{align}
&i \psi_t = -\frac12 \De \psi +V \psi+ \kappa |\psi|^{p-1}\psi \qquad (t,x,y)\in [0,T]\times [a,b]\times [c,d]\label{e:3.1}\\
& \psi(0,x,y)= \psi_0(x,y) \notag\\
&\psi(t,a,y) = \psi(t,b,y), \;  \psi(t,x,c) = \psi(t,x,d);
\quad \psi_x(t,a,y) =\psi_x(t,b,y), \;  \psi_x(t,x,c) = \psi_x(t,x,d) . \notag
\end{align}
The algorithm is implemented based on time-splitting trigonometric spectral approximations with fine mesh grids and time steps.  The solutions are computed mainly using GPELab \cite{GPElab}
adapted to various cases where $V(x,y)= {(\pm \ga_1^2 x^2\pm \ga^2_2 y^2)}/2$, $\kappa>0$ or $\kappa<0$.
The initial data is taken as either a gaussian in $C^\iy(\R^2)$ or a hat function in $H^1(\R^2)$.

We summarize  the numerical results in Figures \ref{f:defoc.VX2p3t2_g} to \ref{f:defoc.Vnegp5t5_hat}
and then provide error analysis in Tables 1 
and 2 
with progressively finer and finer mesh sizes and time steps. These errors are relatively very small and yield quite high accuracy.
 Relevant error estimates can be found in 
 \cite{Lubich2008splitting} and equation (26) in \cite{LuM13}.
Corresponding to two type of nonlinear regimes, we will select $\kappa = 1$ 
 and  
 $\kappa = -1.9718$ for repulsive and attractive 
  interactions, respectively. Let $a=c=-8,b=d=8$. Let $h=\De x=\De y=(b-a)/M$
   be the meshgrid size and $\De t=T/N$ the time step. 






 A. {Defocusing case: $\kappa=1>0$.} \quad
Set the initial data $\psi_0(x,y)=g_\sigma(x,y)$ with $\sigma=1$, where
\begin{equation}\label{e:g-IC}
g_\sigma(x,y):= \frac{1}{\sqrt{\sigma \pi}}e^{-(x^2 + y^2)/2\sigma} .
\end{equation}
   The following show the figures for the numerical solution $\psi_{approx}$ of \eqref{e:3.1}
 on $R=[-8,8]^2$ at different times in the presence of anisotropic quadratic potentials.
The numerical results are in consistence with the theory that attractive $V$ confines the waves to the ground state while the repulsive $V$ enhances the dispersion or scattering.


\begin{figure}[H]
        \centering
        \begin{subfigure}[b]{.45\textwidth}
                \centering
                \includegraphics[width=\textwidth]{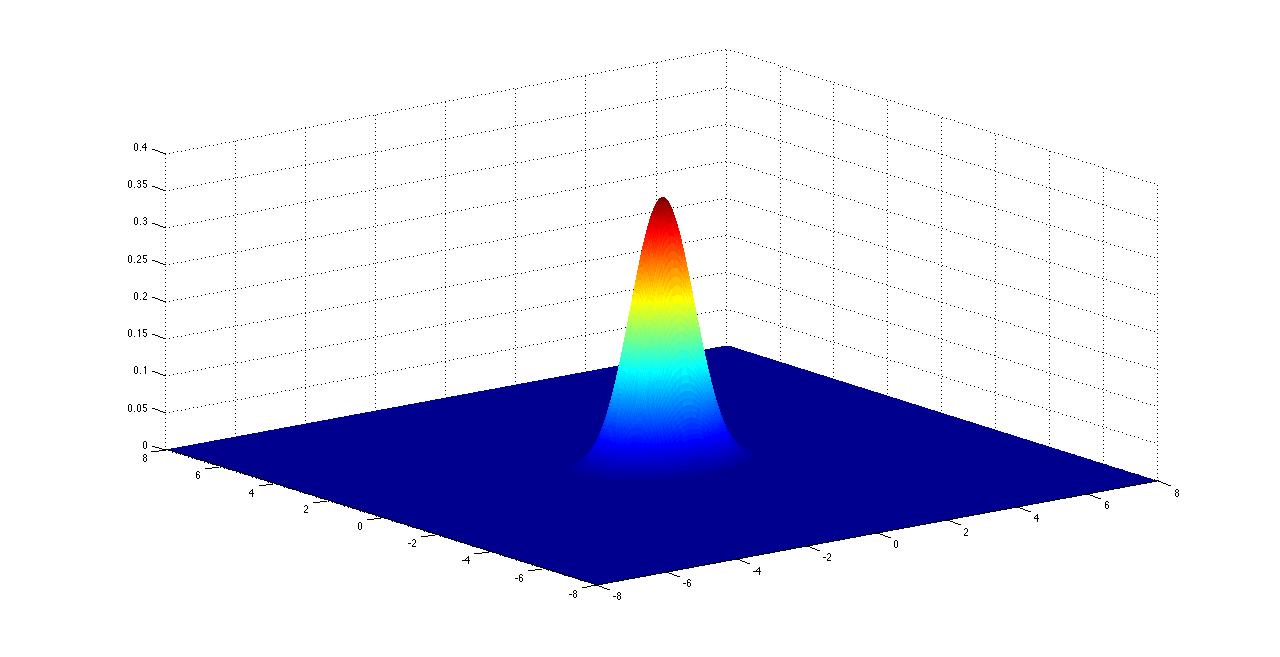}
                \caption{3d view of $|\psi(t,x,y)|^2$ at $t=2$}
                \label{fig:V1M512dt01t283d}
        \end{subfigure}
        \begin{subfigure}[b]{.45\textwidth}
                \centering
                \includegraphics[width=\textwidth]{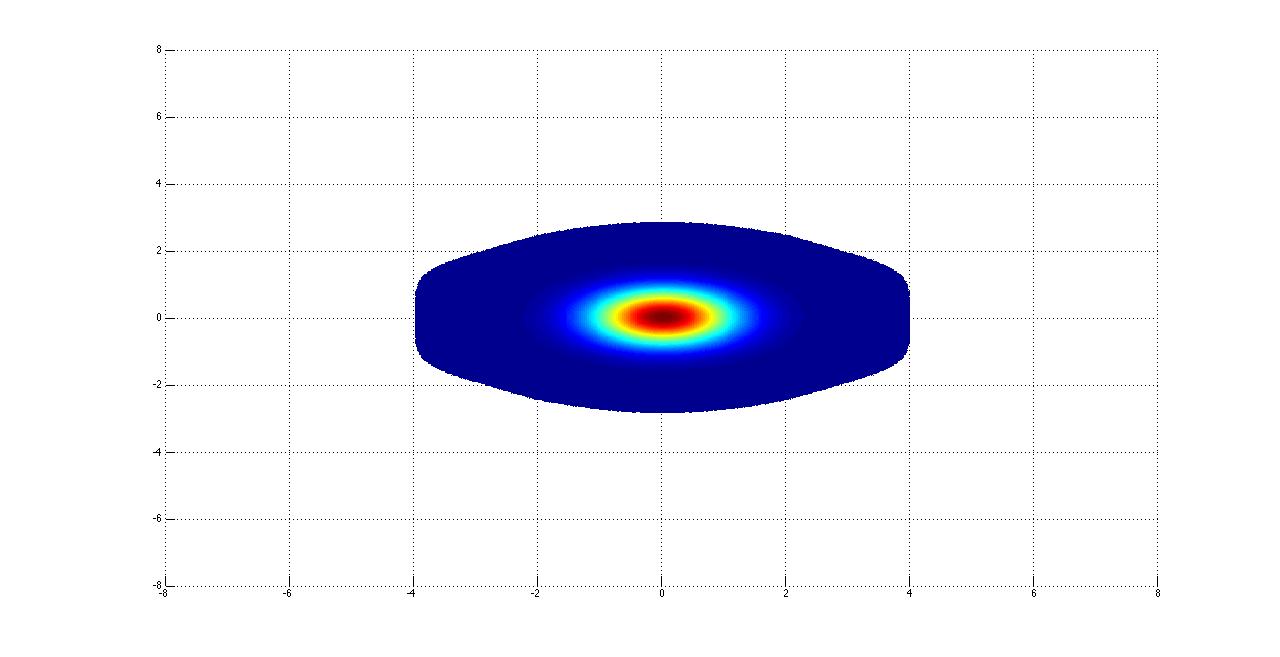}
                \caption{Top view of $|\psi(t,x,y)|^2$ at $t=2$}
                \label{fig:V1M512dt01t28topView}
        \end{subfigure}
                \caption{$p=3$, Defocusing $\kappa=1$, $V=\frac{x^2+4y^2}{2}$ 
                 ($\Delta t=0.01$, $h =\frac{1}{32}$)} 
                 \label{f:defoc.VX2p3t2_g}
\end{figure}

\begin{figure}[H]
        \centering
        \begin{subfigure}[b]{.45\textwidth}
                \centering
                \includegraphics[width=\textwidth]{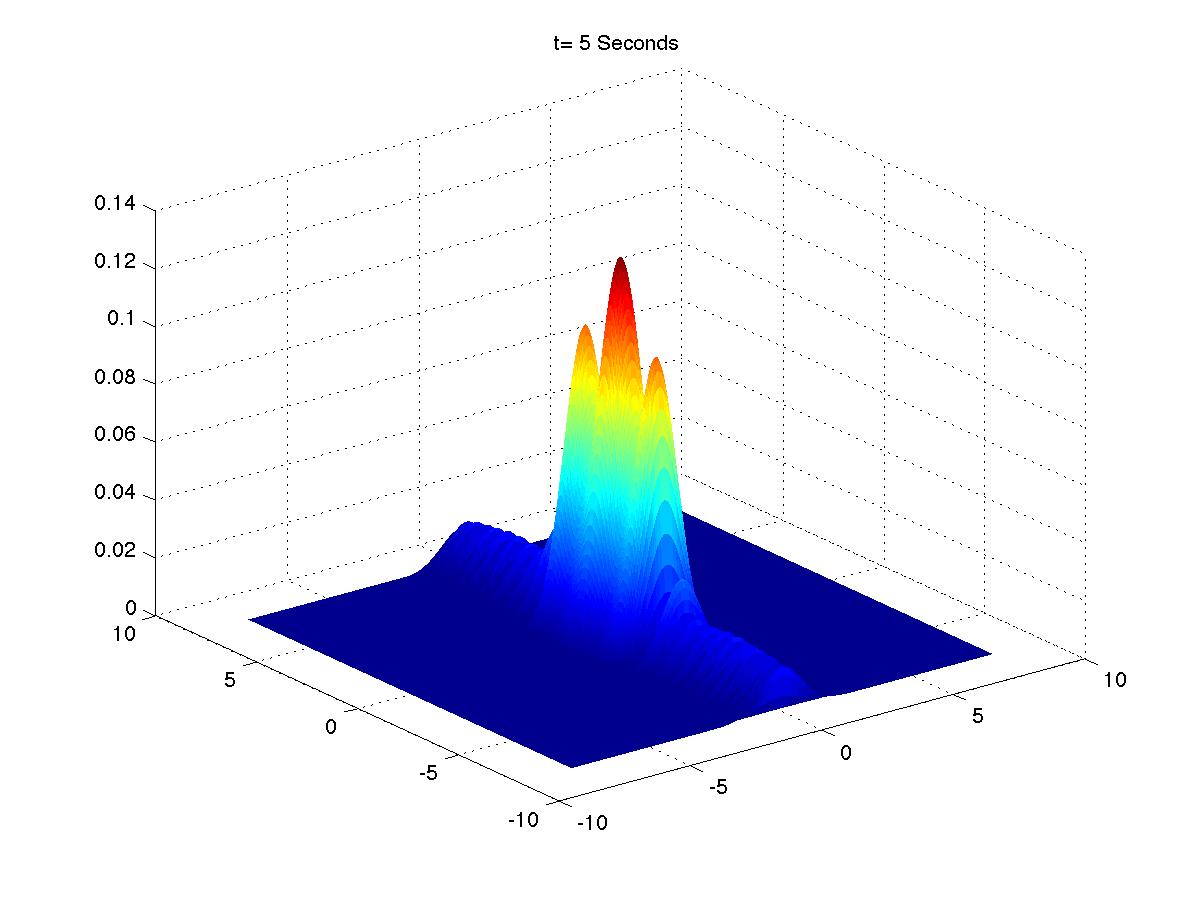}
                \caption{3d view of $|\psi(t,x,y)|^2$ at $t=5$}
                \label{fig:V3t5M256dt00183d}
        \end{subfigure}
        \begin{subfigure}[b]{.45\textwidth}
                \centering
                \includegraphics[width=\textwidth]{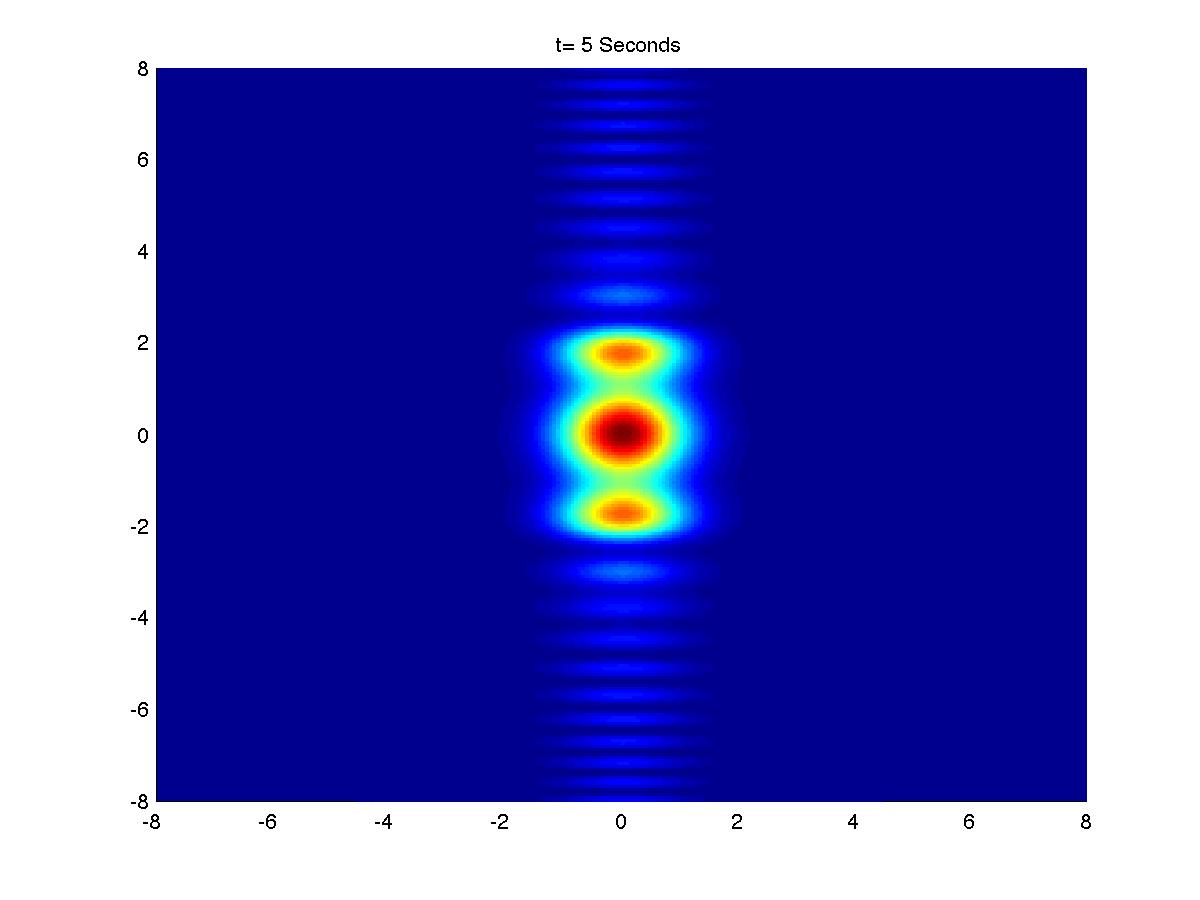}
                \caption{Top view of $|\psi(t,x,y)|^2$ at $t=5$}
                \label{fig:V3t5M256dt0018top}
        \end{subfigure}
                \caption{$p=3$, Defocusing $\kappa=1$, $V=\frac{x^2-y^2}{2}$  \\
                There exists evident dispersion in the $y$-direction} 
                 \label{fig:V3t5M256dt0018}
\end{figure}

\begin{figure}[H]
        \centering
        \begin{subfigure}[b]{.45\textwidth}
                \centering
                \includegraphics[width=\textwidth]{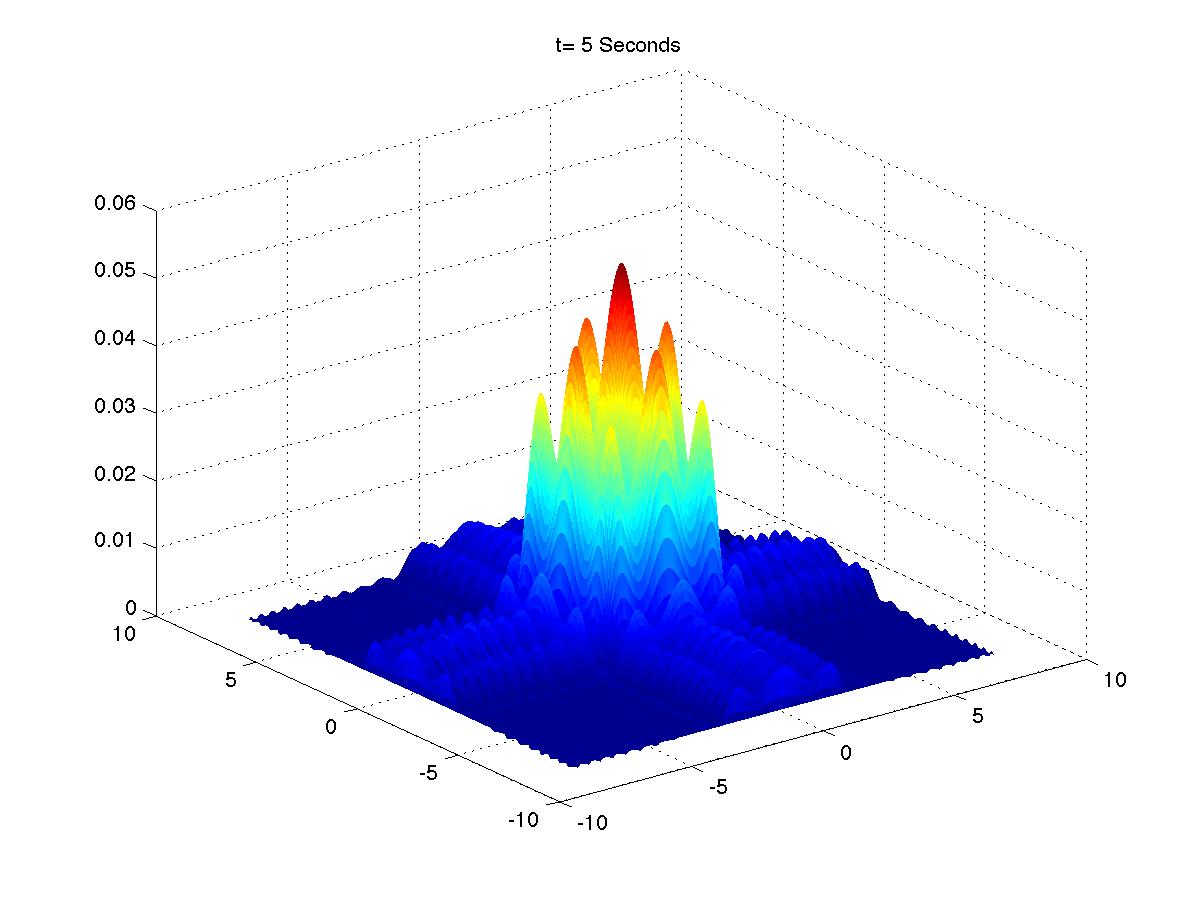}
                \caption{3d view of $|\psi|^2$ at $t=5$} 
                \label{fig:V4t5M256dt00183d}
        \end{subfigure}
        \begin{subfigure}[b]{.45\textwidth}
                \centering
                \includegraphics[width=\textwidth]{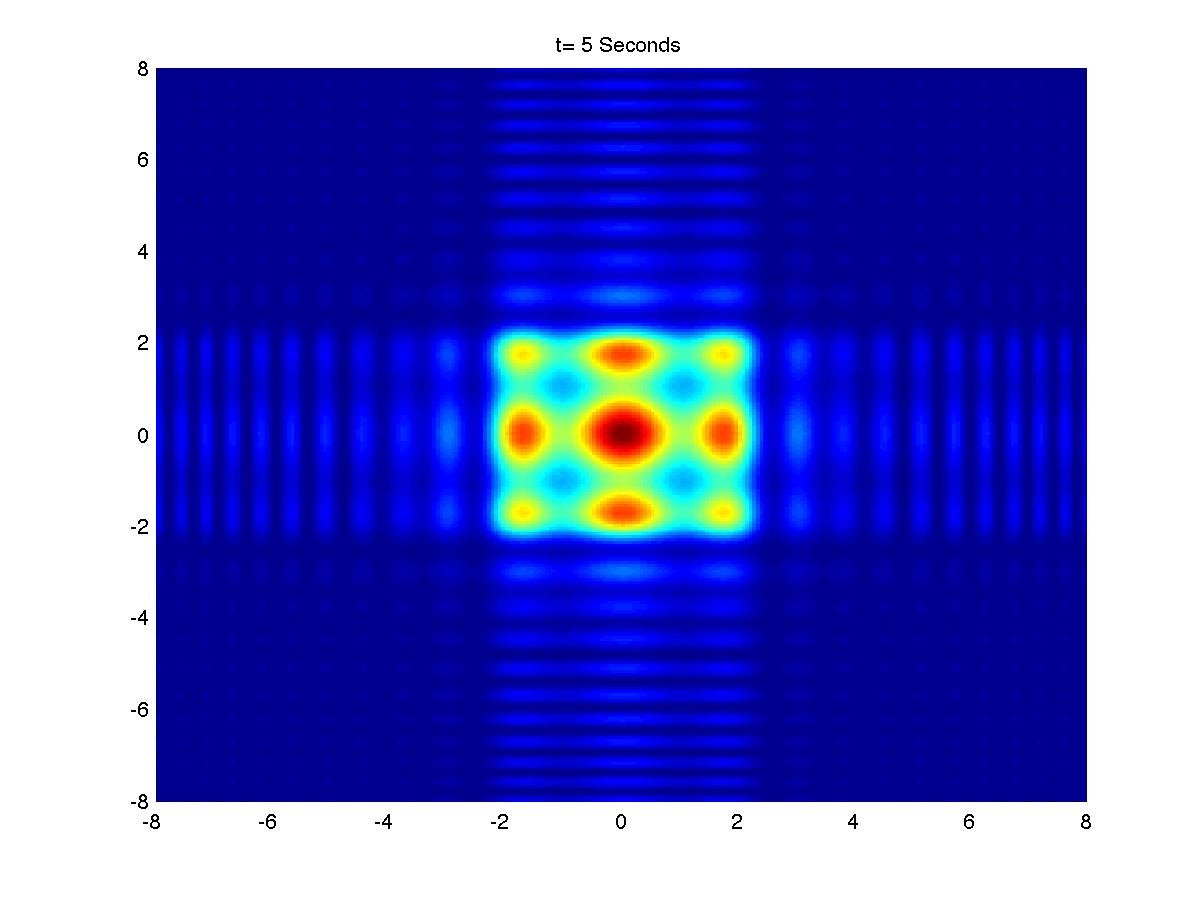}
                \caption{Top view of $|\psi|^2$ at $t=5$}
                \label{fig:V4t5M256dt0018top}
        \end{subfigure}
                \caption{$p=3$, Defocusing $\kappa=1$, $V=-\frac{x^2+y^2}{2}$ ($\Delta t=0.01$, $h =\frac{1}{16}$) \\ 
                Dispersion exist in both $x$- and $y$-directions} 
\end{figure}

B. Focusing case: $\kappa=1.9718<0$ with the same gaussian initial data (\ref{e:g-IC}).  \quad
In the mass-critical case $p=1+4/n=3$, the focusing NLS may have finite blowup solution.
The physics dictates that a positive harmonic potential is  attractive and confines the cooled bosonic atoms. 
 On the other hand, an inverted (negative) harmonic potential is repulsive and supports the dispersion which offsets the impact of the focusing effect.

 \begin{figure}[H]
        \centering
        \begin{subfigure}[b]{.450\textwidth}
                \centering
                \includegraphics[width=\textwidth]{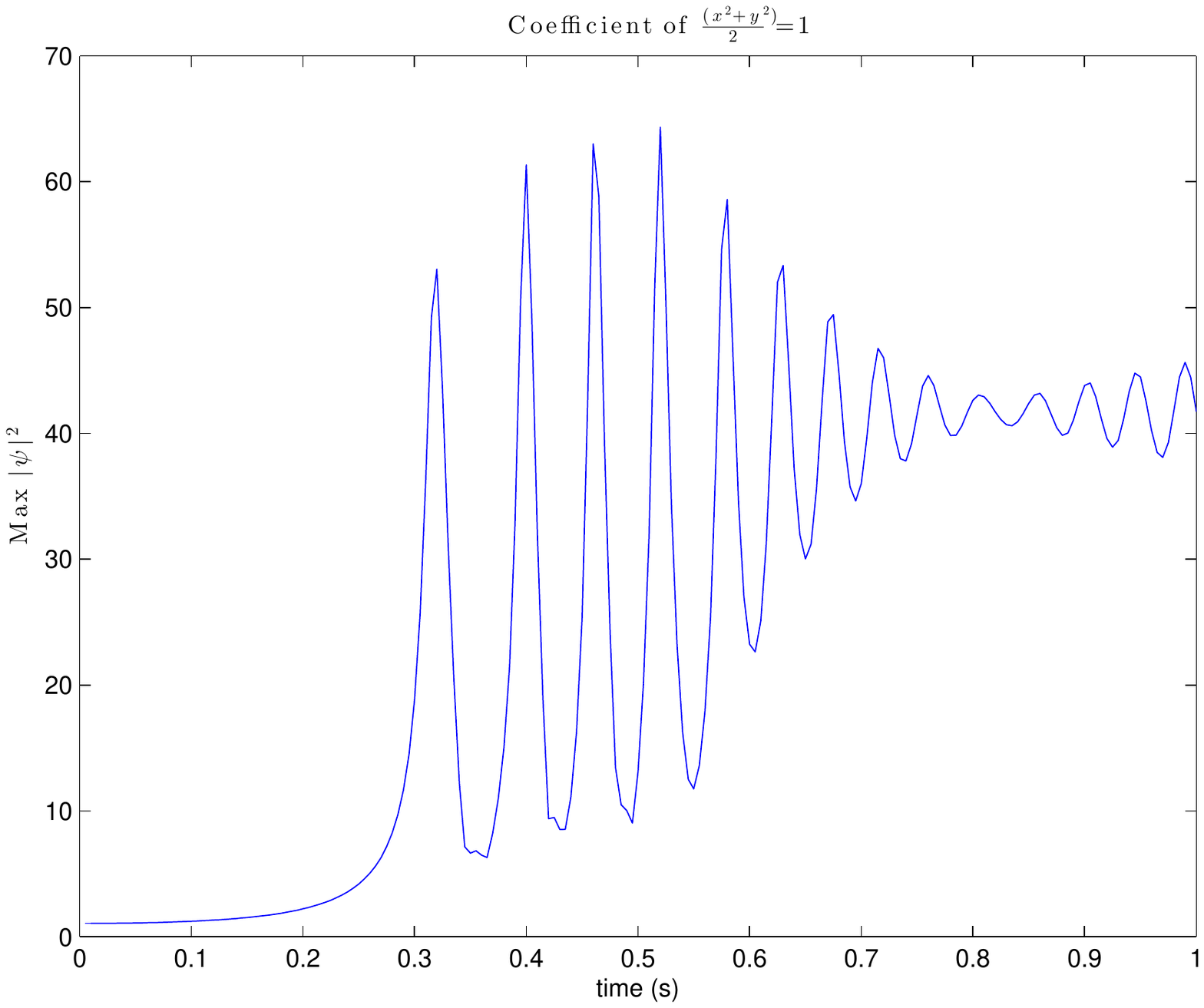}
                \caption{Focusing $\kappa$,\;$V=\frac{x^2+y^2}{2\varepsilon}$, $t\in[0,1]$}
                \label{fig:graphofneg1.pdf}
        \end{subfigure}
        \begin{subfigure}[b]{.450\textwidth}
                \centering
                \includegraphics[width=\textwidth]{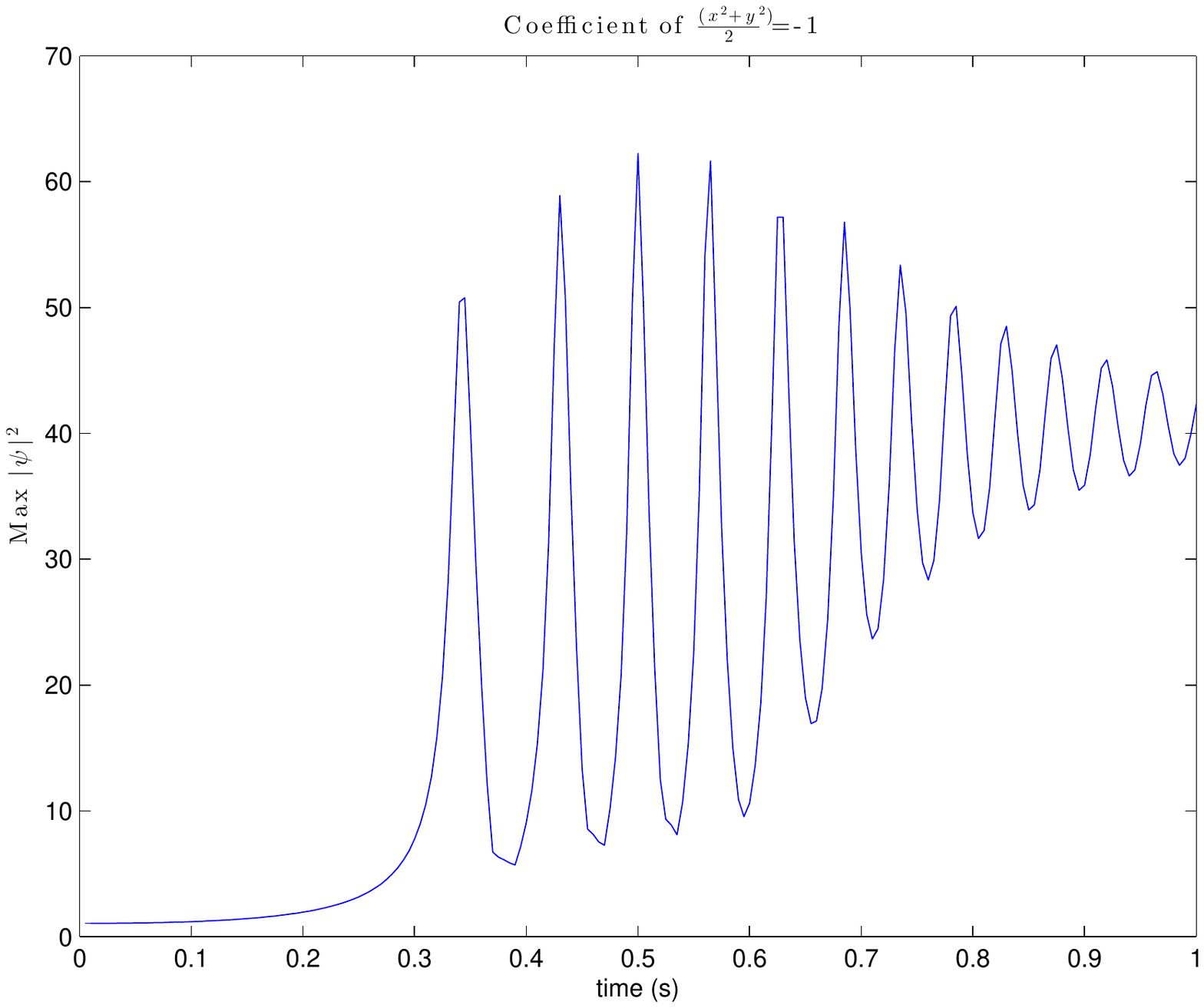}
                \caption{Focusing $\kappa$,\; 
                $V=-\frac{x^2+y^2}{2\varepsilon}$, $t\in[0,1]$}
                \label{fig:Focust1Vcoeffneg1.pdf}
        \end{subfigure}
               \caption{$\max_{(x,y)} |\psi|^2$ vs time, $\psi_0=\textrm{gaussian}$
                ($p=3$, $\kappa = -1.9718$,
                $\varepsilon=0.3$, 
                $\Delta t=0.01$, $h=\frac{1}{32}$)}
                 \label{f:Vneg1blowup}
\end{figure}
Now we observe from Figure \ref{f:Vneg1blowup} that if $V$ changes from $(x^2+y^2)/{2\varepsilon}$ to $-(x^2+y^2)/{2\varepsilon}$,
 then the blowup time has  a slight delay at approximately $t=0.35$.

 \begin{figure}[!b]
        \centering
        \begin{subfigure}[b]{.45\textwidth}
                \centering
                \includegraphics[width=\textwidth]{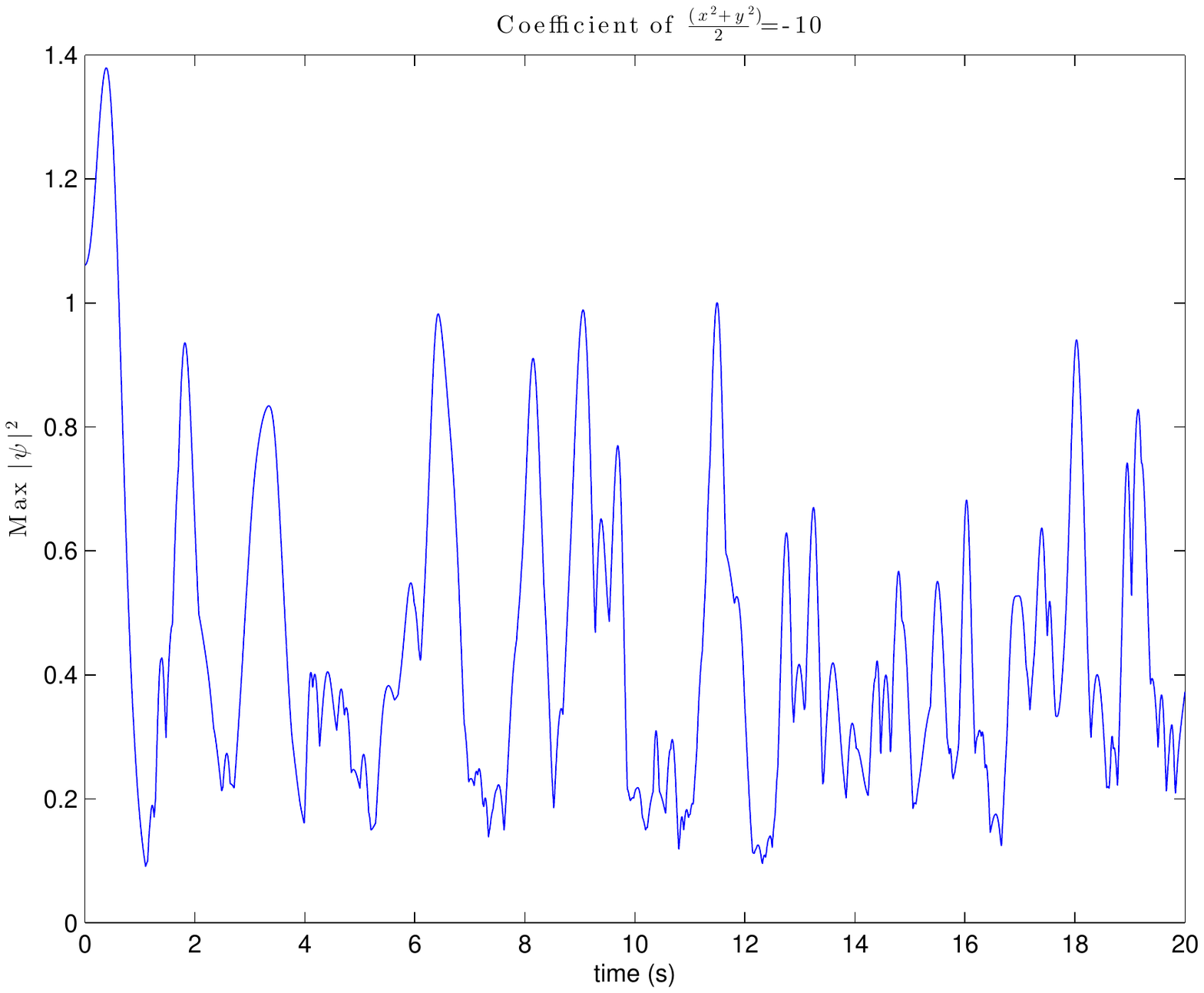}
                \caption{$\psi_0=\textrm{gaussian}$,  $V=\frac{-5(x^2+y^2)}{\varepsilon}$, $t\in[0,20]$}
                \label{f:graphofneg10.pdf}
        \end{subfigure}
        \begin{subfigure}[b]{.45\textwidth}
                \centering
                \includegraphics[width=\textwidth]{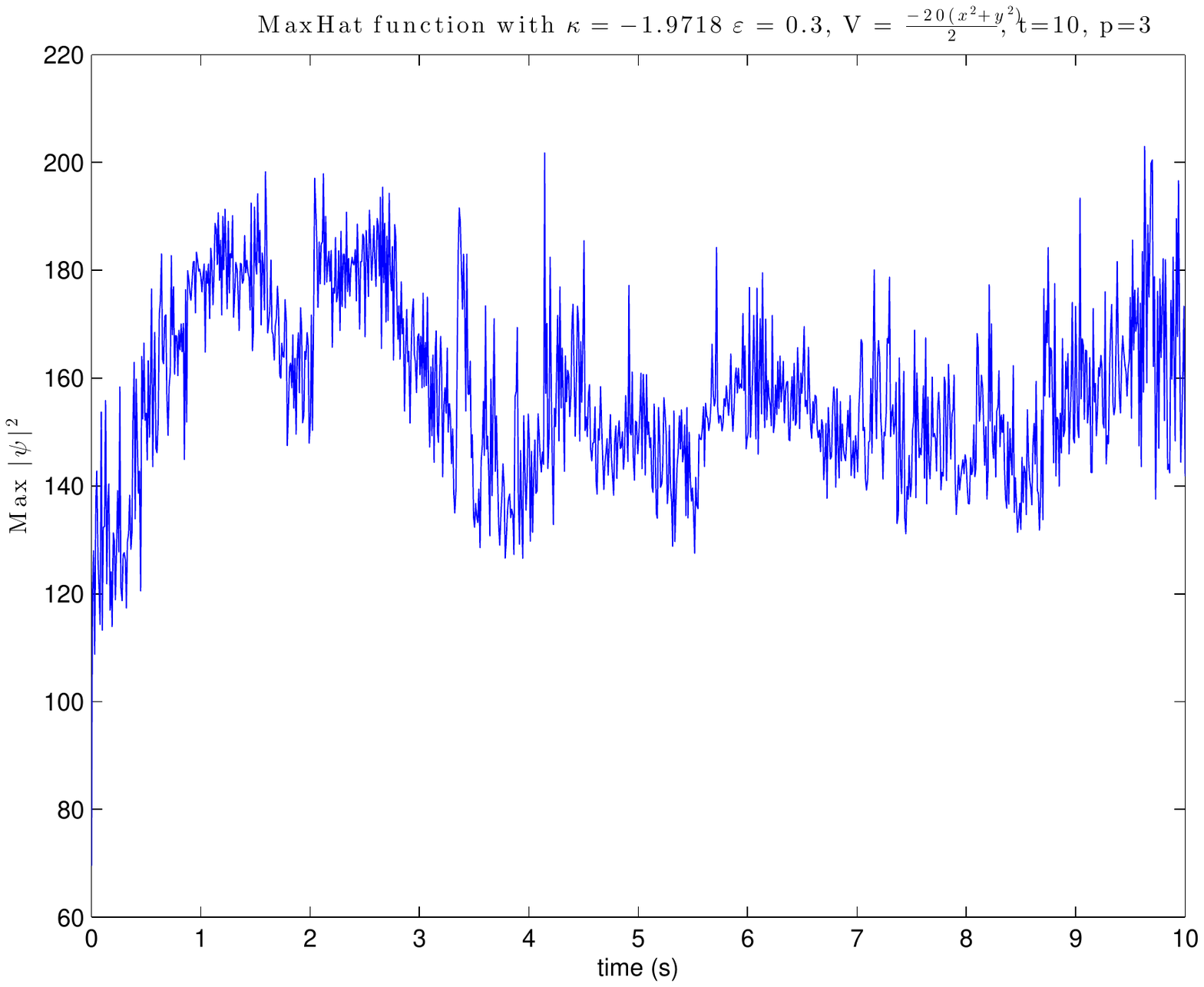}
                \caption{$\psi_0=\textrm{hat function}$, $V=\frac{-10(x^2+y^2)}{\varepsilon}$, $t\in[0,10]$}
                \label{fig:Focust20Vcoeff10.pdf}
        \end{subfigure}
                \caption{$\max_{(x,y)} |\psi|^2$ vs time $t$  ($\kappa=-1.9718$, $\varepsilon=0.3,  p=3$)}\label{f:Vneg10blowup}
\end{figure}

C. Initial data equal to the ``hat function'' $\psi_0=h\in H^1(\R^2)$. \quad
In Figures \ref{f:graphofneg10.pdf} and \ref{fig:Focust20Vcoeff10.pdf} we compare the density function
$|\psi|^2$ for two different initial data,
 one is given by the gaussian \eqref{e:g-IC}
and the other is given by the ``hat function''
\begin{align*}
&h(x,y)=
(8-|x|)(8-|y|) .
\end{align*}
The solutions show that if the magnitude of the frequency  is large enough then the inverted harmonic potential $V$ counteracts the focusing effect which leads to  global in time existence.
On a quite long time interval, they both reveal self-similarity (``multifractal-like'') for the density function 
although $h\in H^1$ has a larger magnitude with low regularity. 
However with $\psi_0$ being the hat function,
$|\psi(t,x,y)|^2$ is more irregular in temporal and much more singular in spatial variables, see Figure
\ref{f:hat.foc-frac.sing}.

The numerical results agree with Theorem \ref{t:gwp-L2smalldata} and \cite[Theorem 3.2]{Z12}.
Note that on local time interval the amplitude of $\psi$ is higher than in the free case $V=0$.
The lack of  the long time decay or scattering may be due to the fact that equation \eqref{e:3.1} has a ``truncation'' version.

 \begin{figure}[H]
        \centering
        \begin{subfigure}[b]{.45\textwidth}
                \centering
                \includegraphics[width=\textwidth]{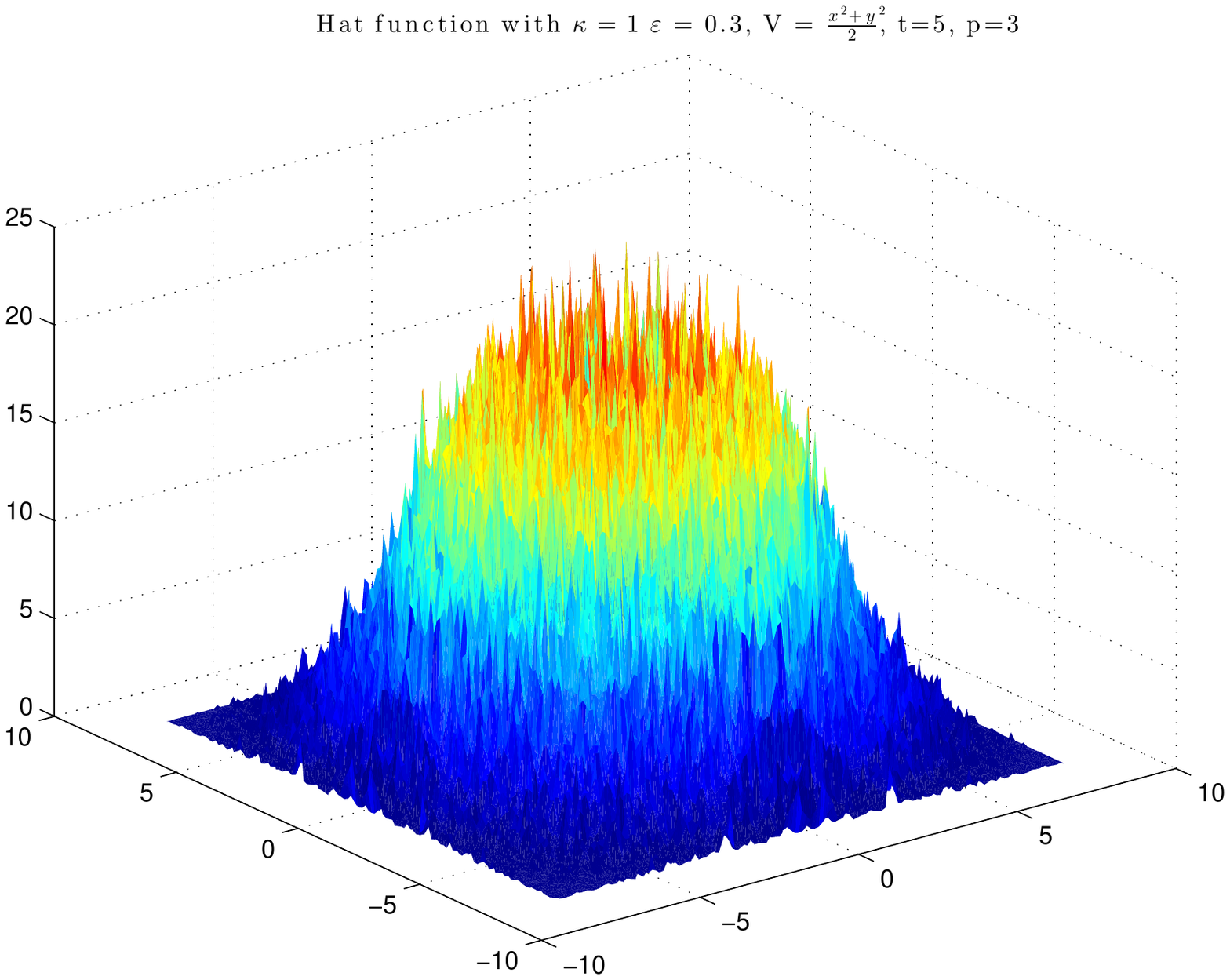}
                \caption{$\kappa=1$, $V=\frac{x^2+y^2}{2\veps}$}
                \label{f:defoc.VX2p3t5}
        \end{subfigure}
        \begin{subfigure}[b]{.45\textwidth}
                \centering
                \includegraphics[width=\textwidth]{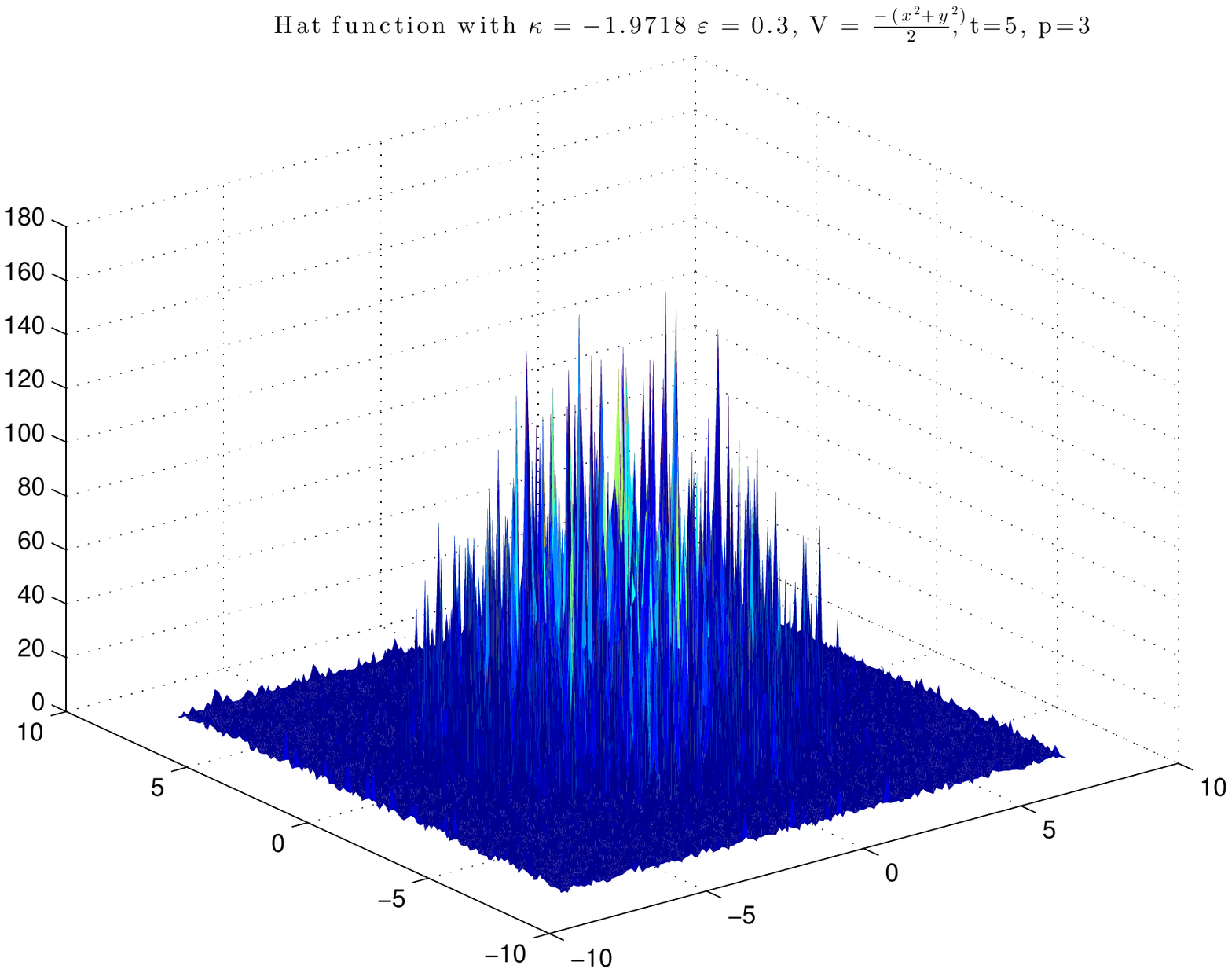}
                \caption{$\kappa=-1.9718$, $V=-\frac{x^2+y^2}{2\varepsilon}$}
                \label{fig:foc.VnegX2p3t5}
        \end{subfigure}
                \caption{$|\psi(t,x,y)|^2$ at $t=5$  ($\psi_0=\textrm{hat function}$, $p=3$, $\varepsilon=0.3$)}\label{f:hat.foc-frac.sing}
\end{figure}

\begin{table}
  \begin{tabular}{|c| c c c  c |}
\hline
Potential $V$  &  $h=\frac{1}{4}$ & $h=\frac{1}{8}$ & $h=\frac{1}{16}$ & $h=\frac{1}{32}$ \\ \hline
0                                                &2.5353e-05 & 1.2107e-11 & 3.4148e-12 & 1.3345e-11 \\ \hline
$\frac{x^2+y^2}{2\varepsilon}$& 1.8215e-05 & 4.0089e-10 & 1.7532e-10 & 8.5637e-11 \\ \hline
$-\frac{x^2+y^2}{2\varepsilon}$&  8.7515e-04 & 9.0129e-06 & 3.3705e-06 & 1.5221e-06 \\ \hline
$\frac{x^2+10y^2}{2\varepsilon}$& 1.7456e-01 & 1.3289e-03 & 7.5563e-10 & 7.9029e-11 \\ \hline
$\frac{x^2-10y^2}{2\varepsilon}$&  2.7557e+00 & 4.8414e+00 & 3.5978e+00 & 5.0977e-02 \\
\hline
\end{tabular}
\caption{Spatial discretization error analysis $\| \psi_{exact} - \psi_{approx} \|_{L^2}$ at $t=1$ on $R=[-8,8]^2$
 (Defocusing $\kappa=1$, $\veps=1$, $\Delta t = 0.00005$, $\psi_0= g_1$)}\label{tab:SpaceError8}
\label{tab:sp.err-defoc.p3t1_g}
\end{table}

\begin{table} 
  \begin{tabular}{|c| c c c c c|}
\hline
Potential $V$ & $\Delta t = 0.01$ & $\Delta t = 0.005$ & $\Delta t = 0.0025$ & $\Delta t = 0.00125$ & $\Delta t = 0.000625$ \\ \hline
0                                                &2.5615e-03 & 6.3592e-04 & 1.5871e-04 & 3.9662e-05 & 9.9139e-06 \\ \hline
$\frac{x^2+y^2}{2\varepsilon}$&1.3647e-02 & 3.4068e-03 & 8.5140e-04 & 2.1283e-04 & 5.3203e-05 \\ \hline
$-\frac{x^2+y^2}{2\varepsilon}$& 4.9640e-02 & 1.2426e-02 & 3.1075e-03 & 7.7695e-04 & 1.9425e-04\\ \hline
$\frac{x^2+10y^2}{2\varepsilon}$& 2.7675e-01 & 6.7647e-02 & 1.6747e-02 & 4.1805e-03 & 1.0447e-03\\
\hline
$\frac{x^2-10y^2}{2\varepsilon}$&  1.3843e+01 & 4.1819e+00 & 1.1328e+00 & 3.1327e-01 & 5.5413e-02\\
\hline
\end{tabular}
\caption{
 Temporal discretization error analysis $\| \psi_{exact} -\psi_{approx} \|_{L^2}$ at $t=1$ on
$R=[-8,8]^2$\\  \quad\quad(Defocusing $\kappa=1$, $\veps=1$, $\De x=\De y=\frac{1}{64}$, $\psi_0=g_1$)}\label{tab:TimeError8}
\end{table}

Tables 1 and 2 show the error analysis between the approximate solution and the exact solution.
In both cases the results have reached good accuracy as well as efficiency.
However when we test on the case  where $\kappa=-1.9718$ ($\veps=0.3$, $\Delta t =0.00005$),
the spatial and the temporal error analysis seem to indicate quite big difference between the use of
 the 
 gaussian and the use of 
 the hat function. 
 Numerical result 
 shows that in the \emph{focusing} case,
 if $\psi_0=h$, then 
the approximation solution $\psi_{approx}$ along with the error $\Vert \psi_{exact}-\psi_{approx}\Vert_{L^2}$
becomes larger in short time and the blowup comes sooner with more
singularities. Note that the relative error is not small since $\Vert \psi(t)\Vert_2=1024/3$.
This might suggests that for numerical purpose one needs to use smoother initial data in order to maintain the prescribed accuracy, see the discussions in \cite{Lubich2008splitting, LuM13}.




\begin{figure}[H]
        \centering
        \begin{subfigure}[b]{.45\textwidth}
                \centering
                \includegraphics[width=\textwidth]{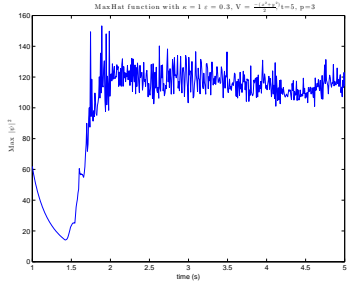}
                \caption{Defocusing cubic NLS ($p=3$)}
                \label{fig:hat.defoc.neg-1}
        \end{subfigure}
        \begin{subfigure}[b]{.45\textwidth}
                \centering
                \includegraphics[width=\textwidth]{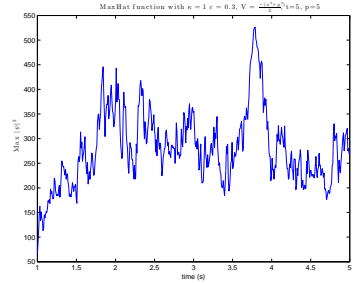}
                \caption{Defocusing quintic NLS ($p=5$)}
                \label{fig:defocV+1}
        \end{subfigure}
                \caption{$\max_{(x,y)}|\psi|^2$ vs time $t\in[0,5]$, $\psi_0=\textrm{hat function}$ ($\kappa = 1$,  $V=-\frac{x^2+y^2}{2\varepsilon}$, $\varepsilon=0.3$)}\label{f:defoc.Vnegp5t5_hat}
\end{figure}

\section{Conclusion}  In the study of the NLS for Bose-Einstein Condensation, the analytic and numerical tools and results we have applied, discovered and reviewed provide good understanding of the modeling equations.
On the numerical aspect,  the Strang splitting method has been shown to be very accurate in many cases
 \cite{bao2002time}. 
In the literature this type of splitting schemes apply to a wide range of nonlinear problems including  KdV, Maxwell-Dirac and Zakharov systems, Boltzmann equation and Landau damping \cite{Baolec,LuM13,CaLi10,EinO12}.  


Recent theory informs that when the quadratic potential has only positive frequency, the wave function exists locally in time and stable, and when $V$ has large negative frequency, then it can counteract the nonlinear effect.
 The outcome of the simulations  agree with the physics of the BEC under trapping conditions. In the case where the anisotropic quadratic potential has sufficiently higher negative coefficients 
   we observe  a dissipative pattern over time, similar to that of the defocusing nonlinearity 
 \cite{Car05,JEMWC96}.  The focusing nonlinearity causes an attractive effect on the condensate that can cause it to ``blowup''. 
  These are true when $\psi_0$ is a gaussian.  In Figure \ref{f:graphofneg10.pdf} 
   after short time 
  the linear $V$ starts to take over and there shows scattering like in the linear periodic case.   
  However,  when the initial data has low regularity, we observe singularities over very short time.
Nevertheless, over a quite long time interval the solutions exhibit multi-fractal structure similar to the linear case \cite{KW94}.
Thus it may be worthwhile to look into the post-blowup behavior of the solutions.

The general understanding is that the BEC mechanism decouples into two states: The ground state from focusing effect and the dispersion from the repulsive interaction. Considering the recent work on BEC with rotation or more generally, the NLS with magnetic fields \cite{LiuT04,AMS10,Z12}, where some questions are quite open,
it would be of interest to continue to study such model under more critical conditions. This investigation, on the analytic and numerical levels, would help explain how the excited sates are formed and how dispersion or scattering can be achieved by appropriately manipulating BEC with potentials, the nonlinearities, and actions of symmetries.


\vspace{.23in}
\nd{\bf Acknowledgment} This work is funded in part by the COSM pilot interdisciplinary project at GSU. The authors would like to thank Dr.~M. Edwards  for constructive discussions.

\end{document}